\documentclass{amsart}
\usepackage{amsmath,amssymb,amsthm,color,url}
\newtheorem{thr}{Theorem}

\newtheorem{claim}[thr]{Claim}

\theoremstyle{definition}

\newtheorem{problem}[thr]{Problem}

\theoremstyle{remark}

\begin{document}

\title{Nonnegative rank depends on the field II}

\author{Yaroslav Shitov}
\address{National Research University Higher School of Economics, 20 Myasnitskaya Ulitsa, Moscow, Russia 101000}
\email{yaroslav-shitov@yandex.ru}


\begin{abstract}
We provide an example of a $21\times 21$ matrix with nonnegative integer entries which can be written as a sum of $19$ nonnegative rank-one matrices but not as a sum of $19$ rational nonnegative rank-one matrices. This gives a solution for a problem posed by Cohen and Rothblum in 1993.
\end{abstract}

\maketitle

Let $A$ be a matrix with nonnegative entries in a field $\mathcal{F}\subset\mathbb{R}$. The nonnegative rank of $A$ with respect to $\mathcal{F}$ is the smallest $k$ such that $A$ is a sum of $k$ nonnegative rank-one matrices with entries in $\mathcal{F}$. We denote this quantity by $\operatorname{Rank}_+(A,\mathcal{F})$ or simply $\operatorname{Rank}_+ A$ if $\mathcal{F}=\mathbb{R}$. Cohen and Rothblum asked the following question in the  foundational paper~\cite{CR} published in 1993.

\begin{problem}\label{prCR}
Is there a rational matrix $A$ such that $\operatorname{Rank}_+(A,\mathbb{Q})\neq\operatorname{Rank}_+(A)$?
\end{problem}
 
Applications motivating this problem include the theory of computation. Vavasis~\cite{Vavas} demonstrates the connection between Problem~\ref{prCR} and algorithmic complexity of nonnegative rank. Another notable application of nonnegative ranks is the theory of extended formulations of polytopes~\cite{Yan}, and the possible lack of optimal rational factorizations is a difficulty in this theory~\cite{Roth}. Kubjas, Robeva, and Sturmfels~\cite{KRS} consider Problem~\ref{prCR} in context of modern statistics; they also give a partial solution of this problem. The work on relaxed versions includes the first part of this paper~\cite{mydep} with an example of a field $\mathcal{F}$ and a matrix $A$ over $\mathcal{F}$ such that $\operatorname{Rank}_+(A,\mathcal{F})\neq\operatorname{Rank}_+(A)$. A similar problem for \textit{positive semidefinite} rank has been solved in~\cite{GFR}.

\bigskip

The aim of this note is to give a positive solution of Problem~\ref{prCR} by proving the result mentioned in the abstract. Assuming that all variables are nonnegative, we denote by $\mathcal{B}(\alpha_1,\ldots,\alpha_n)$ and $\mathcal{C}(a_1,a_2,b,c,d)$ the matrices
$$
\begin{pmatrix}
\alpha_1&\ldots&\alpha_n&1&1&1&1\\
1&\ldots&1&1&1&0&0\\
0&\ldots&0&0&1&1&0\\
0&\ldots&0&0&0&1&1\\
0&\ldots&0&1&0&0&1
\end{pmatrix},\,\,\,\,
\begin{pmatrix}
d&2&2&1&0\\
1&2&1&0&1\\
0&0&1&b&0\\
0&1&0&0&c\\
0&1&1&a_1&a_2
\end{pmatrix},
$$
and by $V$ the bottom-right $4\times4$ submatrix of $\mathcal{B}$. The lower bounds in Claims~2--4 can be easily checked for the conventional rank function, and it is sufficient for the proof as the inequality $\operatorname{Rank}\leqslant\operatorname{Rank}_+$ shows. Claim~5 is a basic example, see~\cite{CR}.

\begin{claim}\label{c3}
$\operatorname{Rank}_+ \mathcal{C}\geqslant3$.
\end{claim}

\begin{claim}\label{c4}
Let $\mathcal{C}_1$, $\mathcal{C}_2$ be the matrices obtained from $\mathcal{C}$ by adjoining a non-zero column of the form $(0\,0\,0\,0\,x)^\top$ and by adjoining a non-zero row of the form $(0\,0\,0\,y\,z)$, respectively. Then $\operatorname{Rank}_+ (\mathcal{C}_1)\geqslant 4$, $\operatorname{Rank}_+ (\mathcal{C}_2)\geqslant 4$.
\end{claim}

\begin{claim}\label{c2}
If $\alpha_1,\ldots,\alpha_n$ are not all equal, then $\operatorname{Rank}_+ \mathcal{B}\geqslant5$.
\end{claim}

\begin{claim}\label{c1}
$\operatorname{Rank}_+ V=4$.
\end{claim}

\begin{claim}\label{c2a}
If $\alpha_1=\ldots=\alpha_n\in[0,1]$, then $\operatorname{Rank}_+ \mathcal{B}=4$.
\end{claim}

\begin{proof}
The first row is a sum of other rows taken with nonnegative coefficients.
\end{proof}

\begin{claim}\label{c3a}
Let $a_1=a_2$. Then $\operatorname{Rank}_+ \mathcal{C}\leqslant3$ iff $a_1=b=c=1+\sqrt{0.5}$, $d=\sqrt{2}$.
\end{claim}

\begin{proof}
Let $\alpha=1+\sqrt{0.5}$. It is easy to check that the equalities $a_1=b=c=\alpha$, $d=\sqrt{2}$ follow from a weaker statement $\operatorname{Rank}\mathcal{C}\leqslant3$. To prove the opposite direction, we check that the rows of $\mathcal{C}(\alpha,\alpha,\alpha,\alpha,\sqrt{2})$ are sums of the rows $(0,1,0,0,\alpha)$, $(0,0,1,\alpha,0)$, $(\sqrt{2},2,\sqrt{2},0,0)$ multiplied by nonnegative coefficients.
\end{proof}

\bigskip

Now we proceed with the example. We claim that the matrix
$$
\begin{pmatrix}
\textcolor{green}{2}&\textcolor{green}{2}&\textcolor{green}{2}&\textcolor{green}{1}&\textcolor{green}{0}&0&0&0&0&0&0&0&0&0&0&0&0&1&1&1&1\\
\textcolor{green}{1}&\textcolor{green}{2}&\textcolor{green}{1}&\textcolor{green}{0}&\textcolor{green}{1}&0&0&0&0&0&0&0&0&0&0&0&0&0&0&0&0\\
\textcolor{green}{0}&\textcolor{green}{0}&\textcolor{green}{1}&\textcolor{green}{2}&\textcolor{green}{0}&0&0&0&0&0&0&0&0&1&1&1&1&0&0&0&0\\
\textcolor{green}{0}&\textcolor{green}{1}&\textcolor{green}{0}&\textcolor{green}{0}&\textcolor{green}{2}&0&0&0&0&1&1&1&1&0&0&0&0&0&0&0&0\\
\textcolor{green}{0}&\textcolor{green}{1}&\textcolor{green}{1}&\textcolor{green}{2}&\textcolor{green}{2}&\textcolor{cyan}{1}&\textcolor{cyan}{1}&\textcolor{cyan}{1}&\textcolor{cyan}{1}&0&0&0&0&0&0&0&0&0&0&0&0\\
0&0&0&\textcolor{cyan}{1}&\textcolor{cyan}{1}&\textcolor{red}{1}&\textcolor{red}{1}&\textcolor{red}{0}&\textcolor{red}{0}&0&0&0&0&0&0&0&0&0&0&0&0\\
0&0&0&0&0&\textcolor{red}{0}&\textcolor{red}{1}&\textcolor{red}{1}&\textcolor{red}{0}&0&0&0&0&0&0&0&0&0&0&0&0\\
0&0&0&0&0&\textcolor{red}{0}&\textcolor{red}{0}&\textcolor{red}{1}&\textcolor{red}{1}&0&0&0&0&0&0&0&0&0&0&0&0\\
0&0&0&0&0&\textcolor{red}{1}&\textcolor{red}{0}&\textcolor{red}{0}&\textcolor{red}{1}&0&0&0&0&0&0&0&0&0&0&0&0\\
0&0&0&0&1&0&0&0&0&\textcolor{blue}{1}&\textcolor{blue}{1}&\textcolor{blue}{0}&\textcolor{blue}{0}&0&0&0&0&0&0&0&0\\
0&0&0&0&0&0&0&0&0&\textcolor{blue}{0}&\textcolor{blue}{1}&\textcolor{blue}{1}&\textcolor{blue}{0}&0&0&0&0&0&0&0&0\\
0&0&0&0&0&0&0&0&0&\textcolor{blue}{0}&\textcolor{blue}{0}&\textcolor{blue}{1}&\textcolor{blue}{1}&0&0&0&0&0&0&0&0\\
0&0&0&0&0&0&0&0&0&\textcolor{blue}{1}&\textcolor{blue}{0}&\textcolor{blue}{0}&\textcolor{blue}{1}&0&0&0&0&0&0&0&0\\
0&0&0&1&0&0&0&0&0&0&0&0&0&\textcolor{yellow}{1}&\textcolor{yellow}{1}&\textcolor{yellow}{0}&\textcolor{yellow}{0}&0&0&0&0\\
0&0&0&0&0&0&0&0&0&0&0&0&0&\textcolor{yellow}{0}&\textcolor{yellow}{1}&\textcolor{yellow}{1}&\textcolor{yellow}{0}&0&0&0&0\\
0&0&0&0&0&0&0&0&0&0&0&0&0&\textcolor{yellow}{0}&\textcolor{yellow}{0}&\textcolor{yellow}{1}&\textcolor{yellow}{1}&0&0&0&0\\
0&0&0&0&0&0&0&0&0&0&0&0&0&\textcolor{yellow}{1}&\textcolor{yellow}{0}&\textcolor{yellow}{0}&\textcolor{yellow}{1}&0&0&0&0\\
1&0&0&0&0&0&0&0&0&0&0&0&0&0&0&0&0&\textcolor{magenta}{1}&\textcolor{magenta}{1}&\textcolor{magenta}{0}&\textcolor{magenta}{0}\\
0&0&0&0&0&0&0&0&0&0&0&0&0&0&0&0&0&\textcolor{magenta}{0}&\textcolor{magenta}{1}&\textcolor{magenta}{1}&\textcolor{magenta}{0}\\
0&0&0&0&0&0&0&0&0&0&0&0&0&0&0&0&0&\textcolor{magenta}{0}&\textcolor{magenta}{0}&\textcolor{magenta}{1}&\textcolor{magenta}{1}\\
0&0&0&0&0&0&0&0&0&0&0&0&0&0&0&0&0&\textcolor{magenta}{1}&\textcolor{magenta}{0}&\textcolor{magenta}{0}&\textcolor{magenta}{1}
\end{pmatrix}
$$
possesses the property mentioned in the abstract. We denote this matrix by $\mathcal{A}$. 

\begin{claim}\label{c5}
$\operatorname{Rank}_+ \mathcal{A}\leqslant 19$. 
\end{claim}

\begin{proof}
Let $\alpha=1+\sqrt{0.5}$, $M_1=\mathcal{C}(\alpha,\alpha,\alpha,\alpha,\sqrt{2})$, $M_2=\mathcal{B}(2-\alpha,2-\alpha)$, $M_3=\mathcal{B}(2-\alpha)$, $M_4=\mathcal{B}(2-\sqrt{2})$. We subtract $M_1$ from the \textcolor{green}{}{green submatrix} of $\mathcal{A}$, and we denote the resulting matrix by $A$. We note that the non-zero entries of $A$ are covered by the disjoint submatrices $A(5,6,7,8,9|4,5,6,7,8,9)$, $A(4,10,11,12,13|5,10,11,12,13)$, $A(3,14,15,16,17|4,14,15,16,17)$, $A(1,18,19,20,21|1,18,19,20,21)$. These matrices are $M_2$, $M_3$, $M_3$, $M_4$ respectively, so we get $\operatorname{Rank}_+ \mathcal{A}\leqslant \operatorname{Rank}_+ M_1+\operatorname{Rank}_+ M_2+2\operatorname{Rank}_+ M_3+\operatorname{Rank}_+ M_4=19$ from Claims~\ref{c2a},~\ref{c3a}.
\end{proof}


\begin{claim}\label{c6}
$\operatorname{Rank}_+(\mathcal{A},\mathbb{Q})\geqslant20$. 
\end{claim}

\begin{proof}
Assuming the converse, we get $\mathcal{A}=A_1+\ldots+A_{19}$ where $A_i$'s are rational rank-one nonnegative matrices. We say that $A_i$ is a \textcolor{red}{red} (or \textcolor{blue}{blue}, \textcolor{yellow}{yellow}, \textcolor{magenta}{magenta}) \textit{summand} if it has at least one non-zero entry with the corresponding color. Since all the entries with these colors are covered by the bottom-right $16\times16$ submatrix of $\mathcal{A}$, and since this submatrix has the form $\operatorname{diag}(\textcolor{red}{V},\textcolor{blue}{V},\textcolor{yellow}{V},\textcolor{magenta}{V})$, we can apply Claim~\ref{c1} and conclude that the colors do not intersect and contain at least four $A_i$'s each. We define $\textcolor{red}{R}$, $\textcolor{blue}{B}$, $\textcolor{yellow}{Y}$, $\textcolor{magenta}{M}$ as the sums of all $A_i$'s that belong to the corresponding colors. We denote by $U$ the sum of uncolored $A_i$'s, and we get $\operatorname{Rank}_+ U\leqslant 3$.

Let us say that \textit{main} entries are those colored \textcolor{cyan}{light blue} or \textcolor{green}{green}. We note that the possible non-zero main entries for \textcolor{red}{red summands} are the \textcolor{cyan}{}{light blue entries} and $\textcolor{green}{}{(5,4)}$, $\textcolor{green}{}{(5,5)}$. For \textcolor{yellow}{yellow summands}, the only possible non-zero main entry is $\textcolor{green}{}{(3,4)}$, for \textcolor{blue}{blue summands} only $\textcolor{green}{}{(4,5)}$ is possible, and for \textcolor{magenta}{magenta summands} only $(1,1)$ is possible. This shows that the \textcolor{green}{}{green submatrix} of $U$ has the form $\mathcal{C}$. Claim~\ref{c3} implies $\operatorname{Rank}_+ U=3$, so we have $\operatorname{Rank}_+ \textcolor{red}{R}\leqslant4$.

If $U$ has a non-zero \textcolor{magenta}{}{light blue entry}, then it contains a submatrix as in Claim~\ref{c4}, which is a contradiction. Otherwise, the red and light blue entries of $\textcolor{red}{R}$ are equal to those of $\mathcal{A}$, which means that the submatrix $\textcolor{red}{R}(5,6,7,8,9|4,5,6,7,8,9)$ has the form $\mathcal{B}$. We see from Claim~\ref{c2} that the $(5,4)$ and $(5,5)$ entries are equal in $\textcolor{red}{R}$; this shows that the $(5,4)$ and $(5,5)$ entries are equal in $U$ as well. We apply Claim~\ref{c3a} and get a contradiction with the rationality of $U$.
\end{proof}

\end{document}